\documentclass[11pt]{article}
\usepackage{amsmath}
\usepackage[margin=1in]{geometry}
\usepackage{amsthm}
\usepackage{comment}
\usepackage{amssymb}
\usepackage{authblk}
\usepackage[style=alphabetic]{biblatex}
\newtheorem{theorem}{Theorem}[section]
\newtheorem*{theoremA}{Theorem A}
\newtheorem*{theoremC}{Theorem C}
\newtheorem*{fact}{Fact}
\newtheorem{proposition}{Proposition}[section]
\newtheorem{lemma}{Lemma}[section]
\newtheorem{corollary}{Corollary}[section]
\newtheorem*{corollaryB}{Corollary B}
\newtheorem{example}{Example}[section]
\newtheorem{remark}{Remark}[section]
\newcommand{\tr}{\textup{tr}}

\AtEndDocument{%
  \par
  \medskip
  \begin{tabular}{@{}l@{}}%
  Vlad Roman, Department of Mathematics, University of Southern California, \\
  Los Angeles, CA 90089-2532, USA \\
    \textit{E-mail address}: \texttt{vladroma@usc.edu} \\
    Partially supported by NSF grant DMS $1901595$
  \end{tabular}}

\title{The commuting variety of $\mathfrak{pgl}_n$}
\author{Vlad Roman}
\date{July 2024}

\begin{document}

\maketitle

\begin{abstract}
    We are considering the commuting variety of the Lie algebra $\mathfrak{pgl}_n$ over an algebraically closed field of characteristic $p >0$, namely the set of pairs $ \{ (A,B) \in \mathfrak{pgl}_n \times \mathfrak{pgl}_n  \mid [A,B]=0 \} $. We prove that if $n=pr$, then there are precisely two irreducible components, of dimensions $n^2+r-1$ and $n^2+n-2$. We also prove that the variety $\{ (x,y) \in GL_n(k) \times GL_n(k) \mid [x,y]=\zeta I \}$ is irreducible of dimension $n^2 +n/d$, where $\zeta$ is a root of unity of order $d$ with $d$ dividing $n$.
\end{abstract}

\section{Introduction}
    Commuting varieties have been studied ever since the 1950s, the first result being attributed to Motzkin and Taussky \cite{Motzkin-Taussky}, who showed that in the case of the Lie algebra $\mathfrak{gl}_n$, the commuting variety is irreducible of dimension $n^2+n$ (over a field of any characteristic). In the 1970s, Richardson \cite{Richardson} generalized the problem and proved that in characteristic zero, for a reductive Lie algebra $\mathfrak{g}$, the commuting variety is irreducible and moreover, its dimension is given by the formula $dim(\mathfrak{g}) + rank(\mathfrak{g})$. A paper of Levy \cite{Levy} extends Richardson’s theorem to a field of good characteristic. Levy used a result of Dieudonné to reduce the problem to distinguished nilpotent elements. In type $A$ the only distinguished nilpotent elements are the regular nilpotents, so in the case of $\mathfrak{pgl}_n$, the key difficulty is that of the regular nilpotent. Our approach here is more elementary and similar to the Motzkin and Taussky result. A paper which treats a related problem is that of Chen and Lu \cite{ChenLu}, who analyze the variety of pairs $(A,B)$ with $AB=qBA$ and describe the irreducible components and their dimensions.

    \par We first study the variety of pairs which have commutator equal to a scalar multiple of the identity and we prove the following theorem.
    \begin{theoremA}
        Let $k$ be an algebraically closed field of characteristic $p >0$ and assume that $p \mid n$. Then the variety $\mathcal{C} = \{ (A,B) \in \mathfrak{gl}_n \times \mathfrak{gl}_n  \mid [A,B]=I \}$ is irreducible of dimension $n^2+r$, where $n=pr$.
    \end{theoremA}

    As a corollary, we obtain the following.
    \begin{corollaryB}
        Let $k$ be an algebraically closed field of characteristic $p >0$ and assume that $p \mid n$. The commuting variety of $\mathfrak{pgl}_n$, $\mathcal{C}_2 (\mathfrak{pgl}_n) = \{ (A,B) \in \mathfrak{pgl}_n \times \mathfrak{pgl}_n  \mid [A,B]=0\}$ has two irreducible components, of dimensions $n^2+n-2$ and $n^2+r-1$, where $n=pr$.
    \end{corollaryB}

In section $2$ we also consider the other Lie algebras in type $A$ and describe an analogous result. In section $3$ we examine the variety of pairs of elements in the group $GL_n(k)$ whose commutator is a scalar multiple of the identity and prove the following.

\begin{theoremC}
    Let $k$ be an algebraically closed field and $\zeta$ be a root of unity of order $d$ with $d \mid n$. The variety $\mathcal{V} = \{ (x,y) \in GL_n(k) \times GL_n(k) \mid [x,y] = \zeta I\}$ is irreducible of dimension $n^2 +n/d$.
\end{theoremC}
The above theorem is closely related to \cite{ChenLu}, while the two previous results, Theorem A and Corollary B are not.

Throughout the paper we will make use of the following general fact about varieties.
\begin{fact}
    Let $k$ be an algebraically closed field and $X$ be a subset of $GL_n(k)$. If $Y \subseteq GL_n(k)$ is irreducible and there exists a dominant map $GL_n(k) \times Y \rightarrow X$, then $X$ is irreducible as well.
\end{fact}

\section{The commuting variety of $\mathfrak{pgl}_n$}
In the following assume that $k$ is an algebraically closed field of characteristic $p$ and that $p \mid n$.
\par We are considering the variety
$$\mathcal{C}_2 (\mathfrak{pgl}_n) = \{ (A,B) \in \mathfrak{pgl}_n \times \mathfrak{pgl}_n  \mid [A,B]=0\}.$$
Let us define the following sets:
$$U_1=\{ (A,B) \in \mathfrak{gl}_n \times \mathfrak{gl}_n  \mid [A,B]=0\},$$
$$U_2(a) = \{ (A,B) \in \mathfrak{gl}_n \times \mathfrak{gl}_n  \mid [A,B]=aI\}$$
and set
$$U_2=\overline{ \bigcup_{a \neq 0} U_2(a) }.$$

\begin{remark}\label{remarkoncomponents}
  Note that if $(p,n)=(2,2)$, it is clear that $U_1$ and $U_2$ are distinct components. If $n>2$, then $\dim(U_1) > \dim(U_2)$ and thus $U_1$ is a component. Since  $\cup_{a \ne 0} U_2(a)$ is an open subvariety of $U_2$ and $\cup_{a \ne 0} U_2(a) \cap U_1 = \O$ (even modulo scalars), $U_2$ is not contained in $U_1$ and so neither subvariety is contained in the other. Therefore, the irreducibility of the $U_i$ implies these are the two components of $\{(A,B) \in \mathfrak{gl}_n \times \mathfrak{gl}_n \} \mid [A,B] = aI \ \text{for some} \ a\}$.  
\end{remark}


Since $U_1 = \mathcal{C}_2(\mathfrak{gl}_n)$ is the commuting variety, we focus on $U_2(a)$ with $a \in k, a \neq 0$. \par Without loss of generality we may assume that $a=1$. Indeed, the map $U_2(1) \rightarrow U_2(a)$ defined by $(A,B) \mapsto (aA,B)$ is an isomorphism.
Thus we consider pairs $(A,B)$ with commutator $AB-BA=I$. 

\par We observe that finding all such pairs corresponds to finding all representations $\rho : A_1 \rightarrow \mathfrak{gl}_n$ where $A_1$ is the first Weyl algebra. This is a non-commutative algebra over $k$ given by generators $x,y$ with relation
$$xy-yx=1,$$ namely
$$A_1=k \langle x,y \rangle /(xy-yx-1).$$
First let us consider the irreducible representations of this algebra. Note that by taking traces in the defining relation we must have
$$0=\tr(xy-yx)=\tr(1)=n,$$
where $n$ is the dimension of the irreducible representation. In particular, if $p \nmid n$, then $A_1$ does not have any nonzero finite-dimensional representations. Hence we must have $p \mid n$. In fact, the following holds.

\begin{lemma}
If $char(k)=p$, then any irreducible representation of $A_1$ is of dimension $p$.
\end{lemma}

\begin{proof}
First, it is known that the centre of $A_1$ is the algebra generated by the two elements $x^p,y^p$, namely $$Z(A_1)=k[x^p,y^p].$$
A proof of this fact can be found in a paper of Revoy \cite{Revoy} for instance.

Let $V$ be an irreducible $A_1$-module. By the above result, the elements $x^p,y^p$ are central, so they act as scalars in the representation:
$$x^p \mapsto \alpha I$$
and 
$$y^p \mapsto \beta I$$
for some $\alpha,\beta \in k$.

Now consider the algebra
$$\hat{A_1}=k \langle x,y \rangle /(xy-yx-1, x^p - \alpha, y^p - \beta) \cong A_1/(x^p - \alpha, y^p - \beta).$$
Note that a basis of this is given by $\{ x^i y^j\}$ for $i,j=0,\dots,p-1$ and so $\dim_k(\hat{A_1})=p^2$. Since $V$ is irreducible, the image of $A_1$ in $\text{End}_k (V)$ must be a simple algebra and so the image is isomorphic to $M_r(k)$, where $r=\dim_k(V)$. Therefore, $r=\dim_k(V) \le p$.

On the other hand, since $\hat{A_1}$ is finite dimensional over $k$, it has a linear representation and by taking traces, there are no modules $V$ with $\dim_k(V)<p$.

Now, because $\dim_k(\hat{A_1})=p^2$  (once we have fixed the scalars), it follows that $\hat{A_1} \cong M_p(k)$ and so $\hat{A_1}$ has a unique irreducible module, which is unique up to isomorphism.

\end{proof}

In the above proof, without loss of generality, we could have chosen the scalar $\delta$ instead of $\delta^p$. Indeed, for $k$ algebraically closed, the Frobenius endomorphism $t \mapsto t^p$ is an  automorphism. 

\par We have shown that any irreducible representation of $A_1$ is of dimension $p$ and they are parametrized by the homomorphism $Z(A_1)=k[x^p,y^p] \rightarrow k$.

Now return to the variety $$\mathcal{C}= U_2(1)=\{ (A,B) \in \mathfrak{gl}_n \times \mathfrak{gl}_n  \mid [A,B]=I\}.$$
First suppose that $n=p$. By the Weyl algebra argument above, for given $\alpha,\beta \in k$, there is a single orbit of pairs $(A,B)$ with $A^p=\alpha^p I $ and $B^p=\beta^p I$. One example of how $A$ and $B$ act is the following.

\begin{example}\label{exampleaction}

$$A=\begin{pmatrix}
0 & 0 & 0 & \dots & \alpha^p \\
1 & 0  & 0& \dots & 0   \\
0 & 1  & 0& \dots & 0 \\
\vdots & \vdots& \ddots  & \ddots & 0 \\
0 & 0& 0  & 1 & 0
\end{pmatrix}, 
B=\begin{pmatrix}
\beta & 1 & 0 & \dots & 0 \\
0 & \beta  & 2& \dots & 0   \\
0 & 0  & \beta& \dots & 0 \\
\vdots & \vdots& \ddots  & \ddots & p-1 \\
0 & 0& 0  & 0 & \beta
\end{pmatrix}$$

\end{example}
For distinct choices of $\alpha,\beta \in k$ we obtain distinct irreducible representations of $A_1$. So up to isomorphism and modulo the choice of $\alpha,\beta \in k$, all irreducible representations are of this form.

\par Using the Jordan-Chevalley decomposition we can write $$A=\alpha I+A_0$$ and  $$B=\beta I+B_0,$$
for some $\alpha,\beta \in k$, where $A_0, B_0$ are nilpotent with a single Jordan block.

In the following, let $n=pr$ for some $r \in \mathbb{N}$. We make the following claim:

\begin{proposition}\label{propositionirreducible}
    $\mathcal{C}$ is an irreducible variety.
\end{proposition}
First, we make the following reduction: consider the same variety $\mathcal{C}$ with the first term $A$ a regular matrix. So let
$$\mathcal{D}=\{ (A,B) \in \mathcal{C}  \mid A \text{ regular}\}.$$

Recall that we call an element $A \in \mathfrak{gl}_n$ \textbf{regular} if it satisfies one of the following equivalent conditions:
\begin{itemize}
    \item A has a cyclic vector, meaning a vector $v$ such that the set $\{ v, Av, \dots, Av^{n-1}\}$ is a basis of $k^n$;
    \item the minimal polynomial of $A$ is equal to its characteristic polynomial;
    \item in its Jordan form, $A$ has at most one Jordan block for each eigenvalue;
    \item the centralizer $C(A)=\{ X \in \mathfrak{gl}_n \mid AX=XA\}$ has $\dim_k (C(A)) = n$;
    \item $C(A)=\{ f(A) \mid f \in k[x]\}$;
    \item $A$ is similar to a companion matrix;
    \item $\dim_k(\ker(A))=1$ (in the case where $A$ is nilpotent).
\end{itemize}

\begin{lemma}\label{opendense}
    The subvariety $\mathcal{D}$ is open and dense in $\mathcal{C}$.
\end{lemma}
    
\begin{proof}
    Note that for any $A \in \mathfrak{gl}_n$ there exists a regular matrix $R$ which commutes with $A$. Without loss of generality, we may assume that $A$ is in Jordan canonical form, so 
    $$A=\begin{pmatrix}
J_1 &  & \\
 & \ddots &  \\
 &  & J_k
\end{pmatrix},$$
where the $J_i$ are the Jordan blocks. For each $i$, let $\lambda_i$ be the corresponding eigenvalue.
Then the matrix $$R=\begin{pmatrix}
(\mu_1 -\lambda_1) I +J_1 &  & \\
 & \ddots &  \\
 &  & (\mu_k - \lambda_k) I +J_k
\end{pmatrix},$$
where the scalars $\mu_i -\lambda_i \in k$ are pairwise distinct, is a regular matrix and it commutes with $A$.
Next, note that for $\alpha \in k$ and $R$ regular commuting with $A$ we have

\begin{equation*}
\begin{split}
[A,B +\alpha R] & = [A,B] + \alpha [A,R] \\
 & = I + O\\
& = I. \\
\end{split}
\end{equation*}

\noindent Thus, for every $\alpha \in k$, $(A,B +\alpha R) \in \mathcal{C}$.
\par Note that the subvariety $Z_{\alpha,\beta}:=\{ \beta A+\alpha R \mid \alpha, \beta \in k\}$ of $\mathfrak{gl}_n$ is irreducible since it is the image of the irreducible $\mathbb{A}_k^2$ under a continuous map. Also, the set of regular matrices in $\mathfrak{gl}_n$ is open and intersects $Z_{\alpha,\beta}$. Therefore, the subset of regular matrices in $Z_{\alpha,\beta}$ must be dense in $Z_{\alpha,\beta}$ because of irreducibility.

\par Hence, every $(A,B) \in \mathcal{C}$ must be contained in the closure of some subset $\mathcal{D}_0$ of the variety $\mathcal{D}$, so $\mathcal{D}$ is open and dense in $\mathcal{C}$.

\end{proof}

We will need the following technical lemma.

\begin{lemma}\label{blocksofsizep}
    Let $char(k)=p$, $A \in \mathfrak{gl}_n$ and suppose $[A,B]=I$ for some $B \in \mathfrak{gl}_n$. Then each Jordan block of $A$ has size a multiple of $p$.
\end{lemma}
A proof of this can be found in \cite{Feinberg} and \cite{Anderson} for instance, but we provide a different one here, which fits better in this context.
\begin{proof}
    Let $\rho: A_1 \rightarrow \mathfrak{gl}_n$ be a representation of $A_1$. Recall from the Weyl algebra argument before that the elements $x^p,y^p$ are central and any irreducible representation of $A_1$ is of dimension $p$.
    \par Notice that $\tr(\rho(x)^p)=\tr(\rho(x))^p$ since the Frobenius map preserves addition. Now consider a composition series of any representation $V$ of $A_1$:
    $$0=V_0 \subset V_1 \subset \dots \subset V_n = V.$$
    The trace is the sum of the traces on each composition factor, so we have
    $$\tr(\rho(x))= \sum_{i=1}^n \tr(\rho_{V_i/V_{i-1}}(x))$$
    for all $x \in A_1$.
    On each composition factor $\rho(x^p)$ and $\rho(y^p)$ are scalars, so their trace is $0$ and thus $\tr(\rho(x)) = \tr(\rho(y))=0$ as well.
    \par Now suppose $\rho(x)$ has a Jordan block of size $m$, not a multiple of $p$. Then $\rho(x)$ commutes with some element $M$ which has $\tr(M) \neq 0$. Let $\rho'(x)=\rho(x), \rho'(y)=\rho(y)+M$ be another representation of $A_1$. Then $[\rho(x),\rho'(y)]=I$, but $\tr(\rho'(y)) \neq 0$, a contradiction.
\end{proof}

Now consider the following elements: 
$$X=\begin{pmatrix}
A_0  & {B_0}^{p-1}  & & & \\
 & A_0   & {B_0}^{p-1} & & \\
 &  & \ddots &  &\\
 & & & \ddots & {B_0}^{p-1} \\
 & &  & & A_0   \\
\end{pmatrix},$$
which has block diagonal elements $A_0$ and entries above the diagonal blocks ${B_0}^{p-1}$ (so there are $r -1$ such blocks) and 
$$Y=\begin{pmatrix}
B_0 &  & \\
 & \ddots &  \\
 &  & B_0
\end{pmatrix}.$$

Here $A_0$ and $B_0$ are the (nilpotent) examples coming from the $p \times p$ case. We note that by construction we have $[X,Y] = I$.

\begin{lemma}
    $X$ is regular.
\end{lemma}

\begin{proof}
    Recall that $A_0$ is nilpotent. We proceed by induction on $r$. For the base case $r=1$ the result is immediate because in the $p \times p$ case $A$ and $B$ are regular and there is a single Jordan block. The critical case is $r=2$. Here $X$ and $Y$ are of the form
    $$X= \begin{pmatrix}
        A_0 & {B_0}^{p-1} \\
        0 & A_0\\
    \end{pmatrix},
    Y= \begin{pmatrix}
        B_0 & 0 \\
        0 & B_0\\
    \end{pmatrix}.$$
Note that $$[X,Y]=\begin{pmatrix}
        A_0 & {B_0}^{p-1} \\
        0 & A_0\\
    \end{pmatrix} 
    \begin{pmatrix}
        B_0 & 0 \\
        0 & B_0\\
    \end{pmatrix}
    - \begin{pmatrix}
        B_0 & 0 \\
        0 & B_0\\
    \end{pmatrix} 
    \begin{pmatrix}
        A_0 & {B_0}^{p-1} \\
        0 & A_0\\
    \end{pmatrix}$$
    $$= \begin{pmatrix}
        A_0 B_0 & {B_0}^{p} \\
        0 & A_0 B_0\\
    \end{pmatrix} - 
    \begin{pmatrix}
        B_0 A_0 & {B_0}^{p}  \\
        0 & B_0 A_0\\
    \end{pmatrix}
      = I,$$
      so by Lemma \ref{blocksofsizep} it follows that all Jordan blocks of $X$ have size a multiple of $p$. Note that either $X$ has a single Jordan block of size $2p \times 2p$ or two Jordan blocks, each of size $p$. Therefore, to see that $X$ is regular, it is enough to show that $X^p \neq 0$.

We claim that $X^p \neq 0$. From here, since all Jordan blocks of $X$ are of size a multiple of $p$, it would follow that $X$ must be a single Jordan block.
Note that
$$X^p = \begin{pmatrix}
        A_0 & {B_0}^{p-1} \\
        0 & A_0\\
    \end{pmatrix}^p= \begin{pmatrix}
        A_0^p & \sum_{k=0}^{p-1} A_0^{k} {B_0}^{p-1} A_0^{p-k-1} \\
        0 & A_0^p\\
    \end{pmatrix} = $$
    $$=\begin{pmatrix}
        0 & \sum_{k=0}^{p-1} A_0^{k} {B_0}^{p-1} A_0^{p-k-1} \\
        0 & 0\\
    \end{pmatrix} = \begin{pmatrix}
        0 & L \\
        0 & 0\\
    \end{pmatrix}.$$
Let $e_i$ denote the $i$-th canonical column vector. Without loss of generality we may assume $A_0$ and $B_0$ are as in Example \ref{exampleaction} with $\alpha = \beta = 0$, namely
$$A=\begin{pmatrix}
0 & 0 & 0 & \dots & 0 \\
1 & 0  & 0& \dots & 0   \\
0 & 1  & 0& \dots & 0 \\
\vdots & \vdots& \ddots  & \ddots & 0 \\
0 & 0& 0  & 1 & 0
\end{pmatrix}, 
B=\begin{pmatrix}
0 & 1 & 0 & \dots & 0 \\
0 & 0  & 2& \dots & 0   \\
0 & 0  & 0& \dots & 0 \\
\vdots & \vdots& \ddots  & \ddots & p-1 \\
0 & 0& 0  & 0 & 0
\end{pmatrix}.$$
Let $0 \neq v \in \ker(A_0)$ (recall that $\ker(A_0)$ is one dimensional since $A_0$ is regular nilpotent). Then we can take $v = e_p$. We notice that 
$$B_0^{p-1}v = (p-1)! e_1$$ and 
$$B_0^{p-1} e_i = 0 \text{ for $i<p$}.$$
Moreover,
$$A_0^s e_i = e_{i+s} \text{ for $i + s \leq p$}$$
and
$$A_0^s e_i = 0 \text{ for $i + s > p$}.$$
Hence, the only nonzero term in $Lv$ occurs for $k = p-1$ and so 
$$Lv = (\sum_{k=0}^{p-1} A_0^{k} {B_0}^{p-1} A_0^{p-k-1})v = A_0^{p-1}B_0^{p-1}e_p = A_0^{p-1}(p-1)!e_1 = (p-1)!v \ne 0,$$ so we obtain that $X^p \neq 0$, as required.

\par Considering the inductive step, we have $$ X = \left(\begin{array}{c|ccccc} 
	A_0 & {B_0}^{p-1} & & & \\ 
    \hline
	0 & A_0 & {B_0}^{p-1} & &\\
     & & \ddots & \ddots &\\
     & &   &A_0 & {B_0}^{p-1}\\
     & &   & & A_0\\
\end{array}\right)
= \begin{pmatrix}
        A_0 & * \\
        0 & \tilde{A_0}\\
    \end{pmatrix},$$
    where $\tilde{A_0}$ has $r$ blocks of size $p$ and $X$ has $r+1$. We are assuming $\tilde{A_0}$ is regular, so that $\dim_k(\ker(\tilde{A_0}))=1$. Let $v = (v_0,v_1,\dots,v_r)^T \in \ker(X)$ (where each component $v_i \in k^p$), so we have
    $$Xv = \begin{pmatrix}
        A_0 v_0 + {B_0}^{p-1} v_1\\
        \tilde{A_0} (v_1,\dots,v_r)^T
    \end{pmatrix}
    = \begin{pmatrix}
        0_{p \times p} \\
        0_{rp \times rp}\\
    \end{pmatrix}.$$
    Since the kernel of the lower block $\tilde{A_0}$ is one-dimensional by assumption, it must be of the form $\{ (\lambda e_p, 0, \dots, 0)^T\ \mid \lambda \in k \}$.
    Then $v_0$ must be of the form $(*,w,0,\dots,0)^T$ and thus we have reduced the problem to the case $r=2$ proved above.
\end{proof}

Now consider
$$X(a_1,\dots,a_r)=\begin{pmatrix}
A_0 + a_1 I & {B_0}^{p-1}  & & & \\
 & A_0 + a_2 I & {B_0}^{p-1} & & \\
 &  & \ddots &  &\\
 & & & \ddots & {B_0}^{p-1} \\
 & &  & & A_0 + a_r I \\
\end{pmatrix},$$

\par and
$$Y=\begin{pmatrix}
B_0 &  & \\
 & \ddots &  \\
 &  & B_0
\end{pmatrix}.$$

Notice that any regular element $X' \in \mathfrak{gl}_n$ with all Jordan blocks of size a multiple of $p$ is similar to some $X$ of this form (depending on the $a_i$). Furthermore, we have $[X,Y]=I$. Since $X$ is regular, its centralizer consists of polynomials in $X$. Therefore, the only solutions to the equation $[X,Y']=I$ are of the form $Y'=f(X)+Y$ for some polynomial $f$. Moreover, since $X = X(0,\dots,0)$ is regular, clearly $X(a_1, \dots, a_r)$ is regular as well for any choice of the $a_i$.
Thus, for a pair $(A,B) \in \mathcal{D}$, $A$ is conjugate to an element $X$ of the form described above with the given $a_i$, so up to conjugacy, all elements in $\mathcal{D}$ are of the form $(X,Y+f(X))$.
\par Now consider the set $$S=\{ (X(a_1,\dots,a_r),Y+f(X(a_1,\dots,a_r))) \mid a_1,\dots,a_r \in k, f \in P \},$$
where $P$ is the set of polynomials of degree at most $pr-1$ and the map
$$F:k^r \times P \rightarrow S,$$
where $$F((a_1,\dots,a_r),f)=(X(a_1,\dots,a_r),Y+f(X(a_1,\dots,a_r))).$$
Note that $F$ is a morphism from affine space and so its image is irreducible. Therefore $S$ is irreducible ($F$ is actually an isomorphism onto $S$).

Moreover, there is an affine map $G: GL_n(k) \times S \rightarrow \mathcal{D}$ given by 
$$G(g,A,B) = (gAg^{-1},gBg^{-1})$$
for $A,B \in S$. Notice that the above map is surjective and since $GL_n(k) \times S$ is irreducible, its image is as well. Since the image contains $\mathcal{D}$, it is dense in $\mathcal{C}$ and so $\mathcal{C}$ is irreducible. This proves Proposition \ref{propositionirreducible}.

In order to compute the dimension of $\mathcal{C}$, consider the generic case, where the matrix $X$ has $r$ distinct eigenvalues, which is open in the Zariski topology. The pair $(A,B) \in \mathcal{C}$ can be conjugated to a pair of the form
$$A(a_1,\dots,a_r)=\begin{pmatrix}
A_0 + a_1 I &   &   \\
 & \ddots &   \\
 &   & A_0 + a_r I 
\end{pmatrix}, 
B(b_1,\dots,b_r)=\begin{pmatrix}
B_0 + b_1 I &   &   \\
 & \ddots &   \\
 &   & B_0 + b_r I 
\end{pmatrix}
$$
and the corresponding module $M$ will split as a direct sum $$M=V_1 \oplus V_2 \oplus \dots \oplus V_r,$$
where the $V_i$ are irreducible of dimension $p$ and they are determined by the scalars $a_i,b_i \in k$, so $V_i=W(a_i,b_i)$. Now, the dimension of the centralizer of $M$ in $GL_n(k)$ is $r$. Consider the map $GL_n \times  k^r \times k^r \rightarrow \mathcal{C}$ given by $$(g, A(a_1, \dots, a_r), B(b_1, \dots, b_r)) \mapsto  (gAg^{-1}, gBg^{-1}).$$  Note that this is a dominant morphism and a generic fibre (with the $a_i$ and the $b_i$ distinct) has dimension $r$ and so we can compute $$\dim(\mathcal{C})=n^2+2r-r=n^2+r$$
and thus obtain Theorem A.

From the above, together with Remark \ref{remarkoncomponents}, we conclude that $\mathcal{C}_2(\mathfrak{pgl}_n)$ has two irreducible components, one of dimension $n^2+n-2$ (the commuting variety) and one of dimension $n^2+r-1$. For the second component, since the variety $\mathcal{C}$ is of dimension $n^2+r$, one more dimension is gained by varying the nonzero scalar $a$ in $[A,B]=aI$ and two dimensions are lost because in the quotient $\mathfrak{pgl}_n$ each fibre is two-dimensional, so then $\tilde{U}_2$, the image of $U_2$ in the commuting variety of $\mathfrak{pgl}_n$ has dimension $\dim(\tilde{U}_2) = n^2+r+1-2=n^2+r-1$.

Hence, we determine the dimension of the image of the commuting variety of $\mathfrak{gl}_n$:
$$\dim(\mathcal{C}_2(\mathfrak{pgl}_n)) = \max(n^2+n-2,n^2+r-1).$$
Let us also notice that since $n=pr \geq 2r$, the maximal component is the one corresponding to the commuting variety (if $(n,p)=(2,2)$, then the two components have the same dimension and this is the only instance where equality occurs).

We end this section by making a few remarks. If $p \mid n$, then there are several isogeny classes of $SL_n(k)$. It is well known that there are at most $3$ distinct Lie algebras and these are $\mathfrak{pgl}_n, \mathfrak{sl}_n$ and if $p^2 \mid n$ there is $\mathfrak{psl_n} \times k$ as well.
The paper of Levy \cite{Levy} includes the result that $\mathcal{C}_2(\mathfrak{sl}_n)$ is irreducible of dimension $n^2+n-2$ and the proof is analogous to the Motzkin and Taussky argument \cite{Motzkin-Taussky}. Indeed, the only modification occurs in Lemma \ref{opendense}, namely whether any element in $\mathfrak{sl}_n$ commutes with a regular element of zero trace. This is indeed true and it is a Zariski closed condition, so the same proof works.
In the remaining case, notice that the images in $\mathcal{C}_2(\mathfrak{pgl}_n)$ of the pairs in $\mathfrak{gl}_n$ which do not commute is already contained in $\mathfrak{psl_n} \times \mathfrak{psl_n}$. This is because of the fact that $U_2$ is contained in $\mathfrak{sl}_n \times \mathfrak{sl}_n$ by the trace zero result.
Note that clearly $\mathcal{C}_2(\mathfrak{psl_n} \times k)$ is isomorphic to $\mathcal{C}_2(\mathfrak{psl_n}) \times k^2$, so it follows that $\mathcal{C}_2(\mathfrak{psl_n} \times k)$ has two irreducible components, of dimensions $n^2 + r-1$ and $n^2 + n - 4$.

\section{Algebraic groups of Type $A$}
Let $k$ be an algebracially closed field and consider the pairs of elements $(x,y) \in GL_n(k) \times GL_n(k)$ such that $[x,y]=x^{-1}y^{-1}xy=\zeta I$. Note that by taking determinants we have $$1=\det [x,y] = \det (\zeta I) = \zeta^n.$$
Fix $\zeta \in k^*$ a root of unity of order $d$ and assume $d \mid n$. We prove the following result.

\begin{theorem}
The variety $\mathcal{V} = \{ (x,y) \in GL_n(k) \times GL_n(k) \mid [x,y] = \zeta I\}$ is irreducible of dimension $n^2 +n/d$.
\end{theorem}

We first notice that the equation $[x,y] = \zeta I$ has a solution if and only if $y^{-1}xy=\zeta x$, meaning that the conjugacy class of $x$ must be invariant under multiplication by $\zeta$.

\begin{lemma}\label{lemmaconjugate}
    Any solution $x$ to the equation $[x,y]=\zeta I$ is conjugate to an element $D$ of the form 
    $$D = \begin{pmatrix}
A &  & & \\
 & \zeta A & &  \\
 &  & \ddots & \\
 & & & \zeta^{d-1}A
\end{pmatrix},$$
where $A \in GL_{n/d}(k)$ and for distinct $i,j$ the eigenvalues of $\zeta^i A$ and $\zeta^j A$ are distinct.
\end{lemma}

\begin{proof}
    First note that since $x$ is conjugate to $\zeta x$ the eigenvalues of $x$ are closed under multiplication by $\zeta$. So choose $a_1, \dots, a_m \in k$ to be orbit representatives for the eigenvalues of $x$. Now let $A$ be the Jordan canonical form of the part of $x$ involving only $a_1, \dots, a_m$. Then $x$ is conjugate to $D$.
\end{proof}

\begin{corollary}\label{corollaryW}
    The variety $\mathcal{W} = \{ x \mid x \text{ is conjugate to } \zeta x \}$ is irreducible of dimension $n^2 +n/d -n$.
\end{corollary}

\begin{proof}
    Let $\mathcal{W}_0$ be the set of all elements $A$ of the form described above. Then $\mathcal{W}_0$ is irreducible and there exists a surjection $GL_n(k) \times \mathcal{W}_0 \twoheadrightarrow \mathcal{W}$, so $\mathcal{W}$ is irreducible as well. Now let $\tilde{\mathcal{W}_0} \subseteq \mathcal{W}_0$ consist of all $x$ such that the corresponding $A$ is diagonalizable with distinct eigenvalues (this is an open condition). Then we have a map $GL_n(k) \times \mathbb{T}_{n/d} \rightarrow \mathcal{W}$ which is dominant and a generic fibre has dimension $n$ (here $\mathbb{T}_{n/d}$ is the algebraic torus). Hence we obtain
    $$\dim(\mathcal{W}) = n^2+ \dim (\mathbb{T}_{n/d}) - \dim (\text{fibre}) = n^2 + n/d -n.$$
\end{proof}

Now let $$\rho = \begin{pmatrix}
0_{n/d} & 0_{n/d} & \cdots & 0_{n/d} &I_{n/d}\\
I_{n/d} & 0_{n/d} & \cdots & 0_{n/d} & 0_{n/d} \\
0_{n/d} & I_{n/d} & \cdots & 0_{n/d} & 0_{n/d} \\
\vdots&  \vdots& \ddots & \vdots& \vdots \\
0_{n/d} & 0_{n/d} &\cdots & I_{n/d} &0_{n/d}
\end{pmatrix}$$ be the block matrix  that cyclic
permutes the blocks of $D$, namely
$$\rho^{-1}D\rho = \begin{pmatrix}
\zeta A &  & & &\\
 &  \zeta^2 A& & & \\
 &  & \ddots & &\\
 & & & \zeta^{d-1} A &\\
  & & & & A \\
\end{pmatrix}$$

and notice that
$$[D,\rho] = D^{-1}(\rho^{-1}D\rho) =  \zeta I.$$

Thus, the solutions to the equation $[D,y]=\zeta I$ are the elements of the form $y=E \rho$, where $E$ is block diagonal with each block centralizing $A$. So $y$ is conjugate by an element in $C(D)$ to $E_1 \rho$, where 
$$E_1 = \begin{pmatrix}
Y &  & & \\
 &  I& &  \\
 &  & \ddots & \\
 & & & I
\end{pmatrix}.$$

Therefore, every pair $(x,y) \in \mathcal{V}$ is conjugate to $(D,E_1 \rho)$, where $E_1 D = D E_1$. Now let $$\mathcal{Q} = \{ (A,B) \in GL_{n/d} \times GL_{n/d} \mid AB=BA\}$$ be the commuting variety of pairs of $n/d \times n/d$ invertible matrices. Next, consider the mapping $(A,B) \mapsto (D(A),E_1(B)\rho)$, where $D(A) = diag(A,I,\dots,I)$ and $E_1(B) = diag(B,I,\dots,I)$. Since $\mathcal{Q}$ is irreducible, its image $R$ must be as well.
Since any point in $\mathcal{V}$ is conjugate to a point in $R$, the map $GL_n(k) \times R \twoheadrightarrow \mathcal{V}$ is surjective, so it follows that $\mathcal{V}$ is irreducible.

Finally, in order to compute $\dim(\mathcal{V})$, consider the projection of $\mathcal{V}$ onto the first component
$$\mathcal{V} \twoheadrightarrow \{ x \mid (x,y) \in \mathcal{V}\}.$$
By Corollary \ref{corollaryW} the dimension of the image is $n^2-n+n/d$. For generic $x$ (for instance $x$ regular), the dimension of the fibre is $\dim C(x) = n$ and so we obtain $$\dim (\mathcal{V}) = n^2+n/d.$$

\section*{Acknowledgements}
I am indebted to my advisor, professor Robert Guralnick, for his guidance throughout writing this paper and his important advice. I would also like to thank the referee for reviewing the paper.

\printbibliography

\end{document}